\documentclass[12pt]{amsart}
\usepackage{tabmac}
\usepackage{amsthm,amssymb}
\usepackage{hyperref}
\usepackage{epsfig}
\usepackage{pstricks}
\usepackage{pst-plot}
\usepackage{pdflscape}
\usepackage{color}

\usepackage[margin=1.1in]{geometry}
\newtheorem*{statement}{Theorem}
\newtheorem*{conjecturalstatement}{Conjecture}
\newtheorem{theorem}{Theorem}[section]

\newtheorem{lemma}[theorem]{Lemma}

\newtheorem{proposition}[theorem]{Proposition}
\newtheorem{prop}[theorem]{Proposition}

\newtheorem{cor}[theorem]{Corollary}
\newtheorem{conjecture}[theorem]{Conjecture}

\theoremstyle{remark}
\newtheorem{remark}{Remark}[section]

\theoremstyle{definition}
\newtheorem{example}[theorem]{Example}

\def\Z{\mathbb{Z}}

\def\uncrossing{\mathrm{uncrossing}}
\def\wt{\mathrm{wt}}
\def\Pr{\mathrm{Pr}}
\def\uPr{\ddddot{\Pr}}
\def\UPr{\ddddot{{\mathcal Pr}}}
\def\L{\mathcal{L}}

\def\det{{\mathrm{det}}}

\def\rest{{\rm rest}}

\def\uqsln{{U_q'({\mathfrak {\hat {sl_n}}})}}

\def\R{{\mathbb R}}

\author{Thomas Lam}
\email{tfylam@umich.edu}
\address{Department of Mathematics,
University of Michigan, 530 Church St., Ann Arbor, MI 48109 USA}
\author{Pavlo Pylyavskyy}
\email{ppylyavs@umn.edu}
\address{Department of Mathematics,
University of Minnesota, 206 Church St. SE, Minneapolis, MN 55455 USA}
\thanks{T.L. was partially supported by NSF grant DMS-0901111, and by a Sloan Fellowship.}
\thanks{P.P. was partially
supported by NSF grant DMS-0757165.} 

\title{Inverse problem in cylindrical electrical networks}

\begin{document}
\begin{abstract}
In this paper we study the inverse Dirichlet-to-Neumann problem for certain cylindrical electrical networks.  We define and study a birational transformation acting on cylindrical 
electrical networks called the electrical $R$-matrix.  We use this transformation to formulate a general conjectural solution to this inverse problem on the cylinder.  
This conjecture extends work of Curtis, Ingerman, and Morrow \cite{CIM}, and of de Verdi\`ere, Gitler, and Vertigan \cite{dVGV} for circular planar electrical
networks. 
We show that our conjectural solution holds for certain ``purely cylindrical'' networks.  Here we apply
the grove combinatorics introduced by Kenyon and Wilson \cite{KW}. 
\end{abstract}
\maketitle

\section{Introduction}

In this paper we consider the simplest of electrical networks, namely those that consist of only resistors.  The electrical properties of such a network $N$ are completely described 
by the {\it response matrix} $L(N)$, which computes the current that flows through the network when certain voltages are fixed at the boundary vertices of $N$.  

\medskip

De Verdi\`{e}re-Gitler-Vertigan \cite{dVGV} and Curtis-Ingerman-Morrow \cite{CIM} studied the {\it {inverse (Dirichlet-to-Neumann) problem}} for {\it circular planar} electrical networks. Specifically, they considered networks embedded in a disk without crossings, with boundary vertices located on the boundary of the disk. The following theorem summarizes their results.

\begin{statement}\
 \begin{enumerate}
\item Any circular planar electrical network is electrically equivalent to some {\emph {critical}} network, which is characterized by its {\emph {medial graph}} being {\emph {lenseless}}  (see \cite[Th\'{e}or\`{e}me 2]{dVGV}).  
\item Any two circular planar electrical networks having the same response matrix can be connected by simple local transformations: series-parallel, loop removal, pendant removal, and star-triangle transformations discussed in Section \ref{ss:trans}.  Furthermore, if both networks are critical, then only star-triangle transformations are required (see \cite[Th\'{e}or\`{e}me 4]{dVGV} or \cite[Theorem 1]{CIM}).
\item The edge conductances of a critical circular planar electrical network can be recovered uniquely from the response matrix (see \cite[Theorem 2]{CIM} or \cite[Th\'{e}or\`{e}me 3]{dVGV}).
\item The response matrices realizable by circular planar networks are the ones having all {\emph {circular minors}} nonnegative (see \cite[Theorem 3]{CIM}).
\item 
The space $Y$ of response matrices of circular planar networks has a stratification by cells $Y = \sqcup C_i$ where each $C_i \simeq \R_{>0}^{d_i}$ can be obtained as the set of response matrices for a fixed critical network with varying edge weights (see \cite[Th\'{e}or\`{e}me 3 and 5]{dVGV} or \cite[Theorem 4]{CIM}).
 \end{enumerate}
\end{statement}

It is an open problem to extend these results to electrical networks embedded in surfaces with more complicated topology.  In this paper we make progress towards understanding the inverse problem for networks embedded in a cylinder. Our first main result is to construct a birational transformation we call the {\it electrical $R$-matrix}.  This transformation acts on the edge weights of a local portion of an electrical network embedded into the cylinder, preserving all the electrical properties of the network (Corollary \ref{cor:resp}).  Furthermore, this electrical $R$-matrix satisfies the Yang-Baxter relation (Theorem \ref{thm:RYB}), and is a close analogue of the ``geometric $R$-matrices'' of affine crystals, to be explained below.

Using the electrical $R$-matrix, we formulate the following general conjecture.  A more precise version is given as Conjecture \ref{conj:gen}.

\begin{conjecturalstatement}
\
\begin{itemize}
\item[(1')]
Any cylindrical electrical network is electrically equivalent to a critical cylindrical electrical network.
\item[(2')]
Any two cylindrical electrical networks $G$ and $G'$ with the same {\it universal response matrices} are connected by local electrical equivalences.  Furthermore, if $G$ and $G'$ are both critical, then only star-triangle transformations, and electrical $R$-matrix transformations are needed.
\item[(3')]
If a cylindrical electrical network is critical, then the edge conductances can be recovered up to the electrical $R$-matrix action.
\item[(5')]
The space $X$ of universal response matrices of cylindrical electrical networks has an infinite stratification by $X = \sqcup C_i$ where each $C_i \simeq \R_{>0}^{d_i} \times \R_{\geq 0}^{e_i}$ is a semi-closed cell that can be obtained as the set of universal response matrices for a fixed critical network with varying edge weights. 
\end{itemize}
\end{conjecturalstatement}

The naive analogue of (4) does not hold -- see Section \ref{ssec:TNN}.  The universal response matrix in the Conjecture is the response matrix of the universal cover of the cylindrical network $G$.  Roughly speaking, it allows us to not only measure the current flowing through the boundary vertices, but also how many times the current has winded around the cylinder.  
It may be possible to formulate this in a more electrically natural way by measuring magnetic fields.  

Thus the key difference between the planar and the cylindrical cases is that even for a critical network on a cylinder the edge conductances may not be uniquely determined from 
the response matrix. This non-uniqueness comes from the existence of the electrical $R$-matrix, the action of which preserves both the underlying graph of the network and the response 
matrix, while changing the edge conductances.  The action of the electrical $R$-matrix can on the one hand be thought of as a Galois group, and on the other hand as a monodromy group.

\medskip

Our second main result is to establish the above Conjecture for a certain class of cylindrical networks we denote $N(m)$.  These critical networks can be thought of as the ``purely cylindrical'' networks.  There is no local configuration for which the star-triangle transformation can be applied in $N(m)$, but the electrical $R$-matrix generates an action of the symmetric group $S_m$.  We show in Theorem \ref{thm:sol} that for the networks $N(m)$ the edge conductances are recovered uniquely from the universal response matrix modulo this $S_m$ action, and in particular the inverse problem has generically $m!$ solutions.  One way to formulate our main conjecture is that the networks $N(m)$ exactly encapsulate the difference between the planar and cylindrical cases.

The proof of Theorem \ref{thm:sol} occupies the technical heart of the paper: we express the edge conductances as limits of certain rational functions of the universal response matrix (Theorem \ref{thm:main}).  Here we use crucially the work of Kenyon and Wilson \cite{KW,KW2}.  Kenyon and Wilson study {\it groves} in circular planar electrical networks.  These are forests whose connected components contain specified boundary vertices.  Kenyon and Wilson connect ratios of grove generating functions with the response matrix of the corresponding network.  By a careful choice of grove generating functions, we can recover the desired edge conductances.

\medskip

Recall that a matrix is {\it {totally nonnegative}} if every minor of it is nonnegative.  There is a mysterious similarity between electrical networks and a different kind of networks 
arising in the theory of { {totally nonnegative matrices}}.  In \cite{LP2}, we presented an approach to understanding this similarity via Lie theory.  Whereas the theory of total nonnegative is intimately related to the class of semisimple Lie groups (\cite{Lus}), we suggested in \cite{LP2} that a different class of ``electrical Lie groups'' is related to electrical networks.  These electrical Lie groups are certain deformed versions of the maximal unipotent subgroup of a semisimple Lie group.

The main ideas of the present work are also motivated by this analogy, though our philosophy here is more combinatorial in nature.  The construction of the electrical $R$-matrix follows the techniques developed in \cite{LP3}.  There we constructed, using purely network-theoretic methods, the geometric (or birational) $R$-matrix of a tensor product of affine geometric crystals for the symmetric power representations of $\uqsln$.  In this paper, we use electrical networks instead of the ``totally nonnegative'' networks of \cite{LP3}, but nevertheless the underlying combinatorics is developed in parallel.  We plan to expand on this analogy in \cite{LP4}.

Our formulation of Conjecture \ref{conj:gen} is also motivated by the analogy with total nonnegativity.  Indeed, analogues of (1')-(5') (and even the missing (4')!) for the {\it totally nonnegative part of the rational loop group} are established in \cite{LP,LP3}.  In particular, in \cite{LP} we studied in detail from the totally nonnegative perspective, the networks $N(m)$, or more precisely, their medial graphs.  Our solution here to the inverse problem for the networks $N(m)$ follows the strategy in \cite{LP}, where edge weights are recovered by taking limits of ratios of matrix entries; this approach was originally applied by Aissen-Schoenberg-Whitney
\cite{ASW} to classify {\it totally positive functions}.  The situation we consider here is technically much more demanding, involving rather intricate Kenyon-Wilson grove combinatorics.

\medskip
\noindent {\bf Acknowledgements.}

\noindent We cordially thank Michael Shapiro for stimulating our interest in the problem.

\section{Electrical networks}
For more background on electrical networks, we refer the reader to \cite{CIM, dVGV, KW}.

\subsection{Response matrix}
For our purposes, an electrical network is a finite weighted undirected graph $N$, where the vertex set is divided into the {\it boundary} vertices and the {\it interior} vertices.  The weight $w(e)$ of an edge is to be thought of as the conductance of the corresponding resistor, and is generally taken to be a positive real number.  A $0$-weighted edge would be the same as having no edge, and an edge with infinite weight would be the same as identifying the endpoint vertices.  

We define the {\it Kirchhoff matrix} $K(N)$ to be the square matrix with rows and columns labeled by the vertices of $N$ as follows:
$$
K_{ij} = \begin{cases} \sum_{e \; \text{joins $i$ and $j$}} w(e) &\mbox{for $i \neq j$} \\
-\sum_{e \; \text{incident to $i$}} w(e) &\mbox{for $i = j$.}
\end{cases}
$$

Let $M$ be a square $n \times n$ matrix, and $I \subset \{1,2,\ldots,n\}$.  Recall that the {\it Schur complement} $M/M_{I,I}$ is the square $n - |I|$ matrix defined to be $M_{[n]-I,[n]-I} - M_{[n]-I,I}M_{I,I}^{-1} M_{I,[n]-I}$, where $M_{J,K}$ denotes the submatrix of $M$ consisting of the rows labeled by $J$ and the columns labeled by $K$.  We define the {\it response matrix} $L(N)$ to be the square matrix with rows and columns labeled by the boundary vertices of $N$, given by the Schur complement
$$
L(N) = K/K_I
$$
where $K_I$ denotes the submatrix of $K$ indexed by the interior vertices.  The response matrix encodes all the electrical properties of $N$ that can be measured 
from the boundary vertices.  Note that our Kirchhoff and response matrices are the negative of those commonly used in the literature.

\subsection{Local electrical equivalences of networks} \label{ss:trans}

We now discuss the local transformations of electrical networks which leave the response matrix invariant.  The following proposition is well-known and can be found for example in \cite{dVGV}.

\begin{figure}[h!]
    \begin{center}
    \input{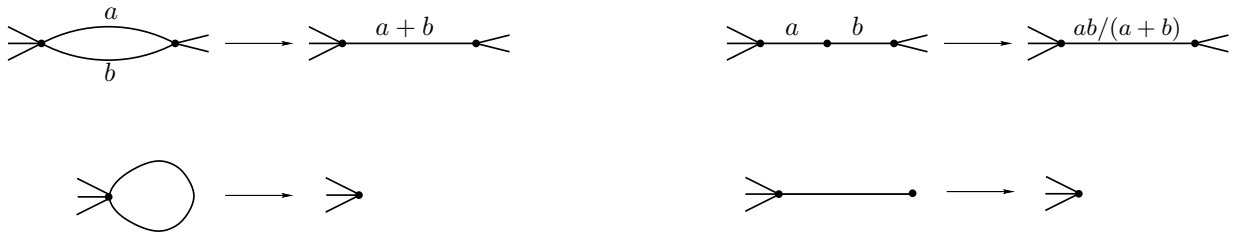}
    \end{center}
    \caption{Series-parallel transformations, loop and pendant removal.}
    \label{fig:eLie1}
\end{figure}

\begin{prop}\label{P:SP}
Series-parallel transformations, removing loops, and removing interior degree 1 vertices, do not change the response matrix of a network.  See Figure \ref{fig:eLie1}.
\end{prop}

\begin{figure}[h!]
    \begin{center}
    \input{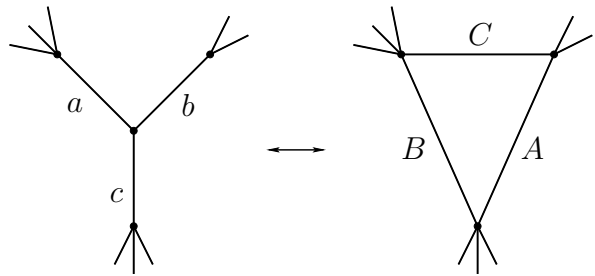}
    \end{center}
    \caption{The $Y-\Delta$, or star-triangle, transformation.}
    \label{fig:elec11}
\end{figure}

The most interesting local transformation is attributed to Kennelly \cite{Ken}.

\begin{prop}[$Y-\Delta$ transformation]
 Assume parameters $a$,$b$,$c$,$A$,$B$ and $C$ are related by 
$$A = \frac{bc}{a+b+c}, \;\; B = \frac{ac}{a+b+c}, \;\; C= \frac{ab}{a+b+c},$$
or equivalently by 
$$a = \frac{AB+AC+BC}{A}, \;\; b= \frac{AB+AC+BC}{B}, \;\; c = \frac{AB+AC+BC}{C}.$$
Then switching a local part of an electrical network between the two options shown in Figure \ref{fig:elec11} does not affect the response of the whole network.
\end{prop}

\subsection{Tetrahedron relation for electrical networks}

\begin{prop}[Tetrahedron relation] \label{prop:FM}
 The sequence of $Y-\Delta$ transformations shown in Figure \ref{fig:elec7} returns the conductances of edges to their original values.
\end{prop}

\begin{figure}[h!]
    \begin{center}
    \input{elec7.pstex_t}
    \end{center}
    \caption{}
    \label{fig:elec7}
\end{figure}

\begin{proof}
 Direct computation.
\end{proof}

\begin{remark}
In the terminology 
of Kashaev, Korepanov, and Sergeev \cite{KKS},  the proposition states that the $Y-\Delta$ transformation is a solution of type ($\epsilon$) to the {\it {functional tetrahedron equation}}.
\end{remark}

\section{Kenyon-Wilson $L$-polynomials}
\label{sec:KWpoly}
\subsection{Groves}
A {\it planar partition} of $[n] = \{1,2,\ldots,n\}$ is a partition $\tau = \{\tau_1,\tau_2,\ldots,\tau_k\}$ such that there do not exist $i < j < k < \ell$ such that $i,k$ belong to the same part of $\tau$, and $k,\ell$ belong to the same part.

A {\it circular planar electrical network} $N$ is an electrical network embedded into a disk, so that the intersection of $N$ with the boundary of the disk is exactly the boundary vertices of $N$.  We suppose that the boundary vertices of $N$ are exactly $[n]$, and that these vertices are arranged in order around the boundary of the disk.  

A {\it grove} in $N$ is a spanning forest where each connected component intersects the boundary.  A grove $\Gamma$ has boundary planar partition $\tau = \{\tau_1,\tau_2,\ldots,\tau_k\}$ if the connected components $\Gamma = \sqcup \Gamma_i$ of $\Gamma$ are such that $\Gamma_i$ contains the boundary vertices labeled by $\tau_i$.  The {\it weight} of the grove is the product of the weights of its edges.  Kenyon and Wilson \cite{KW} study the probability $\Pr(\tau)$ that a random grove of $N$ is of type $\tau$, where the probability of a grove is proportional to its weight.

Let ``$\uncrossing$'' denote the partition of $[n]$ into singletons.

\begin{theorem}\label{T:upr} \cite[Theorem 1.1, Lemma 4.1]{KW}
Let $G$ be a finite circular planar electrical network.
\begin{enumerate}
\item
The ratio $\uPr(\tau):=\Pr(\tau)/\Pr(\uncrossing)$
is an integer-coefficient polynomial in the $L_{ij}(G)$, homogeneous of degree $n-\# \text{parts of}\,\,\tau$.
\item
Suppose $\tau$ be a planar partition with parts of size at most two.  Then the polynomial of (1) depends only on $L_{ij}(G)$, for $i,j$'s which are not isolated parts of $\tau$.
\end{enumerate}
\end{theorem}

\subsection{A bound on certain Kenyon-Wilson polynomials}

Recall from \cite{KW} that for a partition $\tau$, we define
\begin{equation}\label{E:forest}
L_\tau = \sum_F \prod_{\{i,j\} \in F} L_{i,j}
\end{equation}
where the sum is over spanning forests $F$ of the complete graph, for which the trees of $F$ span the parts of $\tau$, and the product is over the edges of $F$.

Let's recall {\bf Rule 1} from \cite[p.5]{KW}.  If a partition $\tau$ is non-planar then one can pick $a < b < c < d$ so that $a,c$ belong to one part of $\tau$, and $b,d$ belong to another part of $\tau$.  Arbitrarily subdivide the part containing $a$ and $c$ into two sets $A$ and $C$ so that $a \in A$ and $c \in C$, and similarly obtain $B$ and $D$.  Denote the remaining parts of the partition $\tau$ by ``rest''.  Then the rule is
\begin{align*}
AC|BD|\rest \to A|BCD|\rest + B|ACD|\rest + C|ABD|\rest+ D|ABC|\rest \\
\qquad -AB|CD|\rest - AD|BC|\rest.
\end{align*}
Iterating this rule transforms each non-planar partition $\tau$ into a linear combination of planar ones, and Kenyon and Wilson show that the coefficients do not depend on how Rule 1 is applied.  Denote by $P_{\sigma,\tau}$ the coefficient of a planar partition $\sigma$ in the application of Rule 1 to $\tau$.  

\begin{theorem}[{\cite[Theorem 1.2]{KW}}]
We have $\uPr(\sigma) = \sum_\tau P_{\sigma,\tau}\,L_\tau$.
\end{theorem}

\begin{lemma}\label{L:singles}
Let $S$ be a cyclic interval.
Suppose that a planar partition $\sigma$ is such that no part $p$ of $\sigma$ contains two elements of $S$.  Let $\tau$ be a possibly nonplanar partition such that $\sigma$ occurs in the expansion of $\tau$ under repeated application of Rule 1 (until no more applications are possible).  Then $\tau$ does not contain any part $p$ such that $|p \cap S| \geq 2$. 
\end{lemma}
\begin{proof}
By \cite[Theorem 1.2]{KW}, the result of repeatedly applying Rule 1 does not depend on the choices made when applying Rule 1.  Let us suppose that $\tau$ contains a part $p$ such that $|p \cap S| \geq 2$.  In applying Rule 1, if we ever encounter that $a,c \in p$, we will first try to choose $A$ and $C$ so that $|A \cap S| \geq 2$ or $|C \cap S| \geq 2$.  This would guarantee that all the partitions occurring in Rule 1 contain some part $p'$ such that $|p' \cap S| \geq 2$.  This choice is impossible only if $a,c \in S$, which would in turn imply that $b$ or $d$ lies in $S$.  But in this case again all the partitions occurring in Rule 1 contain some part $p'$ such that $|p' \cap S| \geq 2$.  Thus $\sigma$ cannot occur in the expansion of $\tau$.
\end{proof}

\begin{lemma}\label{L:combtype}
Suppose $\sigma$ is a partition such that the non-singleton parts of $\sigma$ contain at most $K$ elements.  Then there is a constant $c_K$ such that when $\uPr(\sigma)$ is expanded as a polynomial in the $L_{ij}$'s the coefficient of any monomial in the $L_{ij}$'s is less than or equal to $c_K$.
\end{lemma}
\begin{proof}
Denote the set of boundary vertices by $S$, and by $T \subset S$ the elements in the non-singleton parts of $\sigma$.  Applying Lemma \ref{L:singles}, we see that every $\tau$, such that $\sigma$ occurs in the Rule 1 expansion of $\tau$, have parts that contain at most one element from each cyclic component of $S - T$.  Since the number of cyclic components is bounded by $K$, we see that there is a bound, depending only on $K$, on the number of combinatorial types of possible $\tau$'s.  Here combinatorial type means the partition one obtains when singletons in $S-T$ are removed, and only the relative orders of the remaining elements are remembered.  It follows that one can find a constant $c_K$ so that the coefficient of $\sigma$ in the Rule 1 expansion of $\tau$ is bounded by $c_K$, for any $\tau$.

Next we note that for the sum of \eqref{E:forest}, the edges of $F$ determine $F$, which in turn determines $\tau$.  So each monomial in the $L_{ij}$'s occurs in at most one $L_\tau$.  Thus the coefficient of any monomial in the $L_{ij}$'s is bounded by $c_K$.
\end{proof}

\section{Electrical $R$-matrix}
The ideas of this section follow closely the calculation of the  ``whurl transformation'' in \cite{LP3}.  In a special case the whurl transformation reduces to the {\it birational}, or {\it geometric $R$-matrix} of certain affine geometric crystals.  This motivates the terminology of an {\it electrical $R$-matrix}.

\subsection{The $R$-transformation}\label{ssec:R}
Fix $n \geq 1$ and $m \geq 1$.  We define an electrical network denoted $N(m)$ which is embedded in a cylinder.  It has $2n$ boundary vertices all lying on the boundary of the cylinder,
 with $1,2,\ldots,n$ on the left boundary component and $1',2',\ldots,n'$ on the right boundary component.  There are $(m-1)n$ internal vertices, denoted $M^{(j)}_{i}$ 
where $1 \leq j \leq m-1$ and $1\leq i \leq n$.   For convenience, the boundary vertices $1,2,\ldots,n$ are denoted $M^{(0)}_i$ and the boundary vertices $1',2',\ldots,n'$ are 
denoted $M^{(m)}_i$.
For each $i = 1,2,\ldots,n$ and $j = 0,1,\ldots,m-1$, we have edges from $M^{(j)}_i$ to $M^{(j+1)}_i$, and from $M^{(j)}_i$ to $M^{(j+1)}_{i-1}$.  
Here all lower indices are taken modulo $n$.

\begin{figure}[h!]
    \begin{center}
    \input{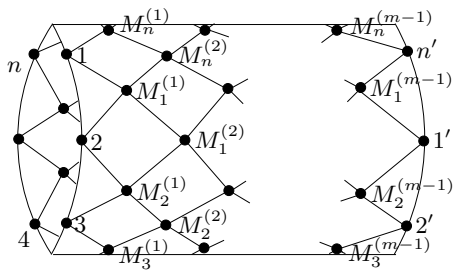}
    \end{center}
    \caption{The cylindrical network $N(m)$.}
    \label{fig:elec15}
\end{figure}

Now we focus on the network $N(2)$.  We label the edge weights of $N(2)$ as follows. For each $i = 1,2,\ldots,n$, we have edges with weights $a_i$ from $i$ to $M^{(1)}_{i-1}$, 
weights $b_i$ from $i$ to $M^{(1)}_{i}$, weights $c_{i+1}$ from $M^{(1)}_i$ to $i'$, and weights $d_{i}$ from $M^{(1)}_{i}$ to $(i-1)'$.  Here all indices are taken modulo $n$.

Now define polynomials $\kappa_i(a,b,c,d)$ and $\tau_i(a,b,c,d)$ as follows:
$$
\tau_i = \sum_{j=0}^{n-1} \left((a_{i+j}+c_{i+j})\prod_{k=i}^{i+j-1}(a_k\,c_k)\prod_{k=i+j}^{i+n-1}(b_k\,d_k) + (b_{i+j}+d_{i+j})\prod_{k=i}^{i+j}(a_k\,c_k)\prod_{k=i+j+1}^{i+n-1}(b_k\,d_k) \right)
$$
and
$$
\kappa_i = (a_{i}+c_{i})\prod_{k=i+1}^{i+n-1}(a_k\,c_k) +  (b_{i}+d_{i})\prod_{k=i+1}^{i+n-1}(b_k\,d_k) 
$$
$$+ \sum_{j=1}^{n-1} \left((a_{i+j}+c_{i+j})\prod_{k=i+1}^{i+j-1}(a_k\,c_k)\prod_{k=i+j}^{i+n-1}(b_k\,d_k) + (b_{i+j}+d_{i+j})\prod_{k=i+1}^{i+j}(a_k\,c_k)\prod_{k=i+j+1}^{i+n-1}(b_k\,d_k) \right)
$$

Also define 
$$
Q = -\prod_i a_i\,c_i + \prod_i b_i \,d_i.
$$

\begin{example}
Suppose $n=2$. Then 
$$\tau_1 = (a_1 + c_1)b_1d_1b_2d_2 + (b_1+d_1)a_1c_1b_2d_2 + (a_2+c_2)a_1c_1b_2d_2 + (b_2+d_2)a_1c_1a_2c_2,$$ 
$$\kappa_1 = (a_1+c_1)a_2c_2 + (b_1+d_1)b_2d_2 + (a_2+c_2)b_2d_2 + (b_2+d_2)a_2c_2.$$

\end{example}

Now introduce additional parallel wires with parameters $p$ and $-p$ from $1$ to $n'$.  This is a special case of the local electrical equivalence for parallel resistors.  (Here the parameters $p$ and $-p$ should be considered formally, instead of as nonnegative real numbers.)  We may perform $Y-\Delta$ operations to move the parameter $p$ through the resistor network, as shown in Figure \ref{fig:elec5}.

\begin{figure}[h!]
    \begin{center}
    \input{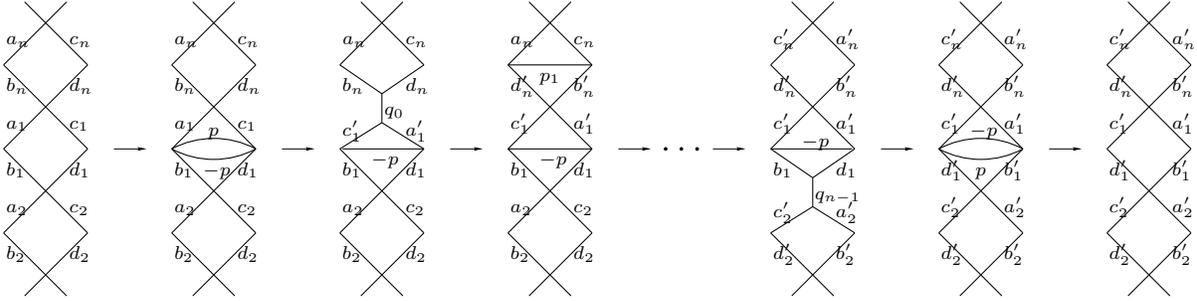}
    \end{center}
    \caption{Calculating the electrical $R$-matrix via certain ``virtual'' electrical networks.}
    \label{fig:elec5}
\end{figure}

\begin{lemma}\label{L:unique}
There is a unique non-zero parameter $p$, which is unchanged after moving through one revolution.
\end{lemma}

\begin{proof}
Let us denote by $p_i$ the value of the ``extra edge'' after $2i$ star-triangle transformations.  (The fourth network in Figure \ref{fig:elec5} shows the location of $p_1$.)  We claim that the parameter $p_i$ is a ratio of two linear functions in the original $p = p_0$. This is easily verified by induction using the star-triangle transformation. Thus after $2n$ star-triangle transformations the equation $p = p_n$ we obtain is either a linear or a quadratic equation, and zero is clearly one of the roots. Therefore there is at most one non-zero solution, and as we shall see soon in the proof of Theorem \ref{T:elecwhirl}, a solution indeed exists.
\end{proof}

Suppose we perform the sequence of $Y-\Delta$ operations of Figure \ref{fig:elec5}, using the parameter $p$ of Lemma \ref{L:unique} (which we still have to prove exists).  Then as illustrated in the final diagram of Figure \ref{fig:elec5}, the ``extra edges'' with parameters $p$ and $-p$ can be removed via the local electrical equivalence for parallel resistors.  We define the {\it electrical $R$-matrix} to be the transformation $R(a_i,b_i,c_i,d_i) = (a'_i,b'_i,c'_i,d'_i)$ induced on the edge weights by this transformation.

\begin{theorem}\label{T:elecwhirl}
The electrical $R$-matrix is given by
\begin{align*}
a'_i =\frac{\tau_{i}}{a_i \kappa_i}\qquad
c'_i = \frac{\tau_{i}}{c_i \kappa_i}\qquad
b'_i =\frac{\tau_{i+1}}{b_i \kappa_{i}}\qquad
d'_i = \frac{\tau_{i+1}}{d_i \kappa_{i}}.
\end{align*}
\end{theorem}
\begin{proof}
We claim that the parameter $p = p_0 = Q/\kappa_1$ is the parameter of Lemma \ref{L:unique}.  Define $p_i$ to be the parameter after $i$ pairs of $\Delta-Y$ and $Y-\Delta$ operations, 
and $q_i$ to be the parameter obtained from $p_i$ by one $\Delta-Y$ operation.  Then
$$
q_i =  \tau_{1-i}/Q \qquad p_i = Q/\kappa_{1-i}.
$$
Indeed, calculating by induction 
\begin{align*}
q_i = \frac{(a_{1-i}+c_{1-i})p_{i}+(a_{1-i}\,c_{1-i})}{p_{i}} 
=\frac{(a_{1-i}+c_{1-i})Q+(a_{1-i}\,c_{1-i})\kappa_{1-i}}{Q}
=\frac{\tau_{1-i}}{Q}
\end{align*}
and similarly for $p_i = b_{1-i}d_{1-i}/(q_{i-1} +b_{1-i}+d_{1-i})$.  We may then calculate that the transformation is given by
\begin{align*}
a'_{1-i} = \frac{a_{1-i}p_{i} + c_{1-i}p_{i}+a_{1-i}c_i}{a_{1-i}} 
=\frac{(a_{1-i}+c_{1-i})Q + a_{1-i}c_{1-i}\kappa_{1-i}}{a_{1-i} \kappa_{1-i}} 
=\frac{\tau_{1-i}}{a_{1-i} \kappa_{1-i}}
\end{align*}
and similarly $c'_{1-i} = \frac{\tau_{1-i}}{c_{1-i} \kappa_{1-i}}$.  We also have
\begin{align*}
b'_{1-i} = \frac{q_{i-1}d_{1-i}}{q_{i-1}+d_{1-i}+b_{1-i}}
=\frac{\tau_{2-i}d_{1-i}}{\tau_{2-i} +(d_{1-i}+b_{1-i})Q} 
=\frac{\tau_{2-i}}{b_{1-i} \kappa_{1-i}}
\end{align*}
and similarly $d'_{1-i} = \frac{\tau_{2-i}b_{1-i}}{\kappa_{1-i}}$.
\end{proof}

\begin{proposition} \label{prop:un}
 The electrical $R$-matrix does not depend on where you attach the extra pair of edges, and is an involution. 
\end{proposition}

\begin{proof}
If one applies the $Y-\Delta$ transformations to successively push through two edges, with weights $q$ and $-q$ negative of each other, then there is no net effect on the weights of the other edges involved.  
Also, the weights on the two edges pushed through remain negatives of each other. Thus if one pushes through $p$ to perform the electrical $R$-matrix, and then a $-p$, the latter will perform a transformation that will undo the first one. 
Furthermore, a working choice of parameter $p$ at one location this way yields a working choice of the parameter at any other location. The uniqueness in Lemma \ref{L:unique} implies that all resulting electrical $R$-matrices are the same. 
\end{proof}

\subsection{Electrical $R$-matrix satisfies Yang-Baxter}

Let us now consider the network $N(3)$.  The procedure of Section \ref{ssec:R} gives two different electrical $R$-matrices acting on $N(3)$: by acting on the part of the network involving vertices $M^{(0)}_i,M^{(1)}_i,M^{(2)}_i$, or by acting on the part of the network involving the vertices $M^{(1)}_i,M^{(2)}_i,M^{(3)}_i$.  We denote these $R$-matrices by $R_{12} \otimes 1$ and $1 \otimes R_{23}$ respectively.

\begin{theorem}\label{thm:RYB}
 The electrical $R$-matrix satisfies the Yang-Baxter equation 
$$(R_{12} \otimes 1) \circ (1 \otimes R_{23}) \circ (R_{12} \otimes 1) =
(1 \otimes R_{23}) \circ (R_{12} \otimes 1) \circ (1 \otimes R_{23}).$$
\end{theorem}

\begin{proof}
First we note that to perform the electrical $R$-matrix, we can either add extra horizontal edges with conductances $p$ and $-p$ between $1$ and $n'$, or we could split the 
vertex $M^{(1)}_1$ into three vertices and add vertical edges with conductances $p$ and $-p$ between them (see Figure \ref{fig:elec16}). 
\begin{figure}[h!]
    \begin{center}
    \input{elec16.pstex_t}
    \end{center}
    \caption{}
    \label{fig:elec16}
\end{figure}

To perform the sequence of $R$-matrices $(R_{12} \otimes 1) \circ (1 \otimes R_{23}) \circ (R_{12} \otimes 1)$, we will add horizontal edges for the first and third factor, but add vertical edges for the second factor.
Let the corresponding weights be $p$, $q$, $r$ as shown in Figure \ref{fig:elec6}.
\begin{figure}[h!]
    \begin{center}
    \input{elec6.pstex_t}
    \end{center}
    \caption{}
    \label{fig:elec6}
\end{figure}
According to Lemma \ref{L:unique} and Theorem \ref{T:elecwhirl} and Proposition \ref{prop:un}, these weights exist and are unique.
Now, apply the $\Delta-Y$ transformation 
as shown in Figure \ref{fig:elec6} to obtain weights $p'$, $q'$ and $r'$, and their negatives on the other side. We claim that if these new weights are pushed around the cylinder, 
they come out the same at the other end. This follows from Proposition \ref{prop:FM}. Now if we push the weights $p$, $q$ and $r$ through and apply the $\Delta-Y$ transformations, we obviously get again the weights $p'$, $q'$ and $r'$. 
Thus by Lemma \ref{L:unique} and Theorem \ref{T:elecwhirl} and Proposition \ref{prop:un}, while pushing $p'$, $q'$ and $r'$ through we are applying 
$(1 \otimes R_{23}) \circ (R_{12} \otimes 1) \circ (1 \otimes R_{23})$. Since, by the electrical tetrahedron relation (Proposition \ref{prop:FM} ) the two results are the same, the claim of 
the theorem follows.
\end{proof}

\begin{cor} \label{cor:braid}
The electrical $R$-matrix gives an action of the symmetric group $S_m$ on $N(m)$.
\end{cor}
\begin{proof}
Follows from Theorem \ref{thm:RYB} and Proposition \ref{prop:un}.
\end{proof}

\section{Electrical ASW factorization}
\subsection{Universal response matrix}
Let $G$ be a cylindrical electrical network.  Thus $G$ is an electrical network embedded  into a cylinder so that the intersection of $G$ with the boundary of the cylinder is exactly the boundary vertices of $G$.  Let $G(\infty)$ denote the universal cover of $G$.  It is an infinite periodic network embedded into an infinite strip.  Given finite sets $V_1 \subset V_2\subset \ldots$ of vertices of $G(\infty)$ which 
eventually cover all vertices of $G(\infty)$, we obtain a sequence of {\it truncations} $G(1)$, $G(2), \ldots$ of $G(\infty)$ as follows.  We let $G(N)$ be the subgraph of $G(\infty)$
 consisting of all edges incident to a vertex in $V_N$.  Furthermore, we declare a vertex of $G(N)$ internal if it lies in $V_N$ and is internal in $G(\infty)$.  Each $G(N)$ is a 
finite planar electrical network.

We suppose that the boundary vertices of $G(\infty)$ are numbered $\Z = \{\ldots,-2,-1,0,1,2,\ldots \}$ on one side of the boundary, and by $\Z'=\{\ldots,-2',-1',0',1',2',\ldots\}$ on the other side of 
the boundary.  (We will always picture the infinite strip as vertical, with vertex labels increasing as we go downwards.)  We define the {\it universal response matrix} of $G$ to be given by $\L_{ij}$ where
$$
\L_{ij} = \lim_{N \to \infty} L_{ij}(G(N))
$$
where $i,j$ denote vertices of $G(\infty)$.  For sufficiently large $N$, any two fixed vertices of $G(\infty)$ will lie in $G(N)$.   These limits exist and are finite, due to following lemma.

If $H$ is an electrical network, and $V$ a subset of its vertices, the {\it response matrix of $H/V$} is the response matrix obtained by declaring all the vertices in $V$ to be interior.

\begin{lemma} \label{lem:Smon}
 Assume $V \subset V'$ are two subsets of vertices of an electrical network $H$, and assume $i$ and $j$ are two vertices not contained in $V'$. Then $L'_{ij}$
in the response matrix of $H/V'$ is at least as large as $L_{ij}$ in the response matrix of $H/V$.
\end{lemma}

\begin{proof}
 The Schur complement with respect to some set $V$ can be taken as a sequence of Schur complements with respect to single vertices in $V$ in some order (see for example \cite[(3.7)]{CM}). Thus,
it would suffice to prove the statement for $V'-V$ consisting of a single vertex. In this case the claim is obvious however, since off-diagonal entries of a response matrix are nonnegative,
and diagonal entries are non-positive.  
\end{proof}

\begin{theorem}\label{T:universalresponse}
The matrix $\L_{ij}$ is a well-defined infinite periodic matrix, which does not depend on which truncations $G(N)$ are taken.
\end{theorem}
\begin{proof}
There are several parts to this statement.

The limits $\L_{ij}$ exist.  Note that for sufficiently large $N$, $L_{ij}(G(N))$ can be calculated by taking $L_{ij}$ of the network $G(\infty)_N$, which is obtained from $G(\infty)$ by declaring that only the internal vertices of $G(N)$ are internal in $G(\infty)_N$.  The network $G(\infty)_N$ is obtained from $G(N)$ by adding extra boundary vertices attached only to boundary vertices of $G(N)$ (and by assuming $N$ is large enough, we may assume that these extra vertices are not incident to $i$ or $j$).  But $L_{ij}(G(N))$ is by definition calculated by measuring the current flowing through $j$ when vertex $i$ is set to one volt and all other boundary vertices are set to zero volts.  Since current does not flow between zero volt vertices it follows that $L_{ij}(G(N)) = L_{ij}(G(\infty))_N$.   By Lemma \ref{lem:Smon} the sequence $L_{ij}(G(N))_N$ is non-decreasing as $N \to \infty$, since each network is obtained from the previous one by declaring some extra vertices internal and taking the corresponding Schur complement.

The limits $\L_{ij}$ do not depend on the sequence of truncations. Assume we have two different sequences $V_1 \subset V_2\subset \ldots$ and $V'_1 \subset V'_2\subset \ldots$. Since we 
know that each eventually covers all vertices, we know that for each $i$ there is a $j$ such that $V_i \subset V'_j$ and $V'_i \subset V_j$. Then applying Lemma \ref{lem:Smon} we conclude
that the two limits bound each other from above, and thus are equal.

The limits $\L_{ij}$ are periodic: $\L_{ij} = \L_{(i+n)(j+n)}$. Indeed, take two sequences $V_1 \subset V_2\subset \ldots$ and $V'_1 \subset V'_2\subset \ldots$, one obtained from the other 
by a shift on the universal cover by the period $n$. We know that they give the same value of $\L_{ij}$ by the previous part. On the other hand, it is clear that the value one of them gives 
for $\L_{ij}$ is the value the other gives for $\L_{(i+n)(j+n)}$.

The limits $\L_{ij}$ are finite. Any truncation of $G(\infty)$ can also be viewed as a truncation of a finite cover $G[m]$ (obtained by lifting $G$ to a $m$-fold cover of the cylinder) for a large enough $m$. By Lemma \ref{lem:Smon} the conductance $L_{ij}(G(N))$ is bounded from above by the same conductance in $G[m]$, which in turn is bounded from above by the same conductance in the original network $G$. Indeed, if $j_1, \ldots, j_m$ are vertices in
$G[m]$ that cover $j$, then $L_{ij}(G) = \sum_{k=1}^m L_{ij_k}(G[m])$.  This can be seen as follows.  Using the linearity of the response and periodicity, the sum $\sum_{k=1}^m L_{ij_k}(G[m])$ measures the current through $j_\ell$ (for any $\ell$) when all vertices $i_k$ that cover $i$ have potential $1$, and all other vertices have potential $0$.  But projecting onto $G$ by identifying all covers of a vertex we see that this current flow is exactly $L_{ij}(G)$.  In particular, $L_{ij}(G(N)) \leq L_{ij_k}(G[m]) \leq L_{ij}(G)$ for any $k$. This shows that $\L_{ij}$ is bounded from above by $L_{ij}(G)$, and thus if the latter is finite, so is the former.
\end{proof}

Thus for fixed $i,j$, we can approximate $\L_{ij}$ arbitrarily well by calculating $L_{ij}(G(N))$ for some large $N$.  

%

\begin{prop}
The universal response matrix $\L_{ij}$ is invariant under the local electrical equivalences of $N$.
\end{prop}
\begin{proof}
For any local electrical transformation in $G$ one can choose a sequence of truncations of $G(\infty)$ containing completely several occurrences of this transformation. 
Since the conductances in these truncations do not change, the limit is also invariant.  
\end{proof}

\begin{cor} \label{cor:resp}
 The electrical $R$-matrix preserves the universal response matrix.
\end{cor}
\begin{proof}
The only comment one needs to make is that the entries of the universal response matrix are limits of rational functions in the edge weights, and that these rational functions make sense formally even when negative conductances are used (as in the derivation of the electrical $R$-matrix).
\end{proof}

\subsection{Circular and cylindrical total nonnegativity}\label{ssec:TNN}
Let $I = \{i_1<i_2<\ldots<i_k\},J = \{j_1 < j_2<\ldots<j_k\} \subset [n]$ be subsets of the same cardinality.  Then $(I,J)$ is a {\it circular pair} if a cyclic permutation of $i_1,i_2,\ldots,i_k,j_k,\ldots,j_1$ is in order.  A $n \times n$ matrix $M$ is {\it circular totally-nonnegative} if the minor $\det(M_{I,J})$ is nonnegative for every circular pair $(I,J)$.  Curtis, Ingerman, and Morrow \cite{CIM} show that the response matrices of circular planar electrical networks are exactly the set of circular totally-nonnegative symmetric matrices for which every row sums to 0.  (Note that the response matrices in \cite{CIM} are the negative of ours.)

Let us extend this to cylindrical electrical networks. Put the total order $\cdots -1<0<1<2< \cdots < 2'<1'<0'<-1' < \cdots$ on $\Z \cup \Z'$.  Let $I, J \subset \Z \cup \Z'$ be two ordered subsets of the same finite cardinality.  Then $(I,J)$ is a {\it cylindrical pair} if a cyclic permutation of $i_1,i_2,\ldots,i_k,j_k,\ldots,j_1$ is in order.  A matrix $M$ with rows and columns labeled (and ordered) with $\Z \cup \Z'$ is {\it cylindrically totally-nonnegative} if the minor $\det(M_{I,J})$ is nonnegative for every cylindrical pair $(I,J)$.

\begin{prop}\label{P:TNN}
Suppose $\L=(\L_{ij})$ is the universal response matrix of a finite cylindrical electrical network.  Then $\L$ is cylindrically totally-nonnegative.
\end{prop}
\begin{proof}
For each fixed cylindrical pair $(I,J)$, and sufficiently large $N$, the truncation $G(N)$ is a finite circular electrical network including all the boundary vertices in $I$ and $J$.  
But then $\det(\L_{I,J}) = \lim_{N \to \infty}\det(L_{I,J}(G(N))) \geq 0$, using Curtis-Ingerman-Morrow's result.
\end{proof}

\begin{remark}
The converse to Proposition \ref{P:TNN}, namely, which cylindrically totally nonnegative matrices are realizable as universal response matrices, is more subtle.  Let $\ldots,j_{-1},j_0,j_1,j_2,\ldots$ be the lifts to the universal cover of a particular vertex in a finite cylindrical electrical network $G$.  Then for a fixed $i$, the (doubly-infinite) sequence $a_k = \L_{i,j_k}$ must satisfy certain recursions or convergence properties.  In the different but closely related setting of total nonnegative points of loop groups, the correct property is to ask for the generating function of $a_k$ to be a rational function (see \cite[Theorem 8.10]{LP3}).
\end{remark}

\subsection{ASW factorization for networks $N(m)$}
We now assume we are given a network $G=N(m)$.  The vertices have been labeled so that if we take the ``low'' edge (from $M_i^{(k)}$ to $M_i^{(k+1)}$) at every step starting from $i$ we will end up at $i'$.  
Note that a shortest path from one side of the cylinder to the other consists of exactly $m$ edges.

From now on we consider groves in the universal cover $G(\infty)$ of $G$ (or in the truncations $G(N)$).  The boundary partition of such a grove would be a planar partition of $\Z \cup \Z'$ arranged on the two edges of an infinite strip (or in the truncations of this).


\begin{lemma}\label{L:short}
Suppose $i$ and $j'$ are can be connected by a path with $m$ edges.  There exists an integer $M$ such that there are no groves $\Gamma$ with the properties
\begin{enumerate}
\item
there is a grove component $\Gamma_{\{i,j'\}}$ with boundary vertices $\{i,j'\}$ and which uses an edge below (resp. above) any of the shortest paths from $i$ to $j'$.
\item
there are grove components with boundary vertices $\{i+1,(j+1)'\},\ldots,\{i+M,(j+M)'\}$ (resp. $\{i-1,(j-1)'\},\ldots,\{i-M,(j-M)'\}$).
\end{enumerate}
\end{lemma}
\begin{proof}
For a grove component $\Gamma_{\{i,j'\}}$ let us call {\it {bad}} the edges it contains that are below the lowest path from $i$ to $j'$. 
\begin{figure}[h!]
    \begin{center}
    \input{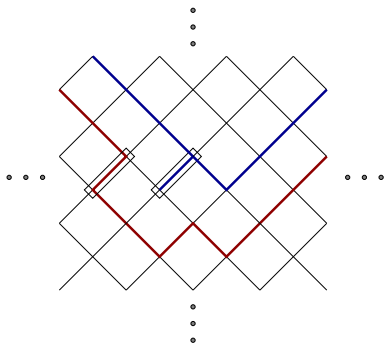}
    \end{center}
    \caption{A bad high edge in the blue grove component forces an earlier bad high edge in the red grove component.}
    \label{fig:elec2}
\end{figure}
Assume the grove component $\Gamma_{\{i,j'\}}$ has bad edges, and furthermore without loss of generality assume that it has a bad high edge. The case of a bad low edge is similar with 
the left and right sides of the network swapped. 
Assume $k$ is the first index such that bad high edge has one of the $M_i^{(k)}$ as its right endpoint. We claim that the $\Gamma_{\{i+1,(j+1)'\}}$ 
component has a bad high edge with right end having index $k-2$ or smaller. Indeed, consider the unique path from $i+1$ to $(j+1)'$ inside $\Gamma_{\{i+1,(j+1)'\}}$. 
In order to avoid touching 
the bad edge of $\Gamma_{\{i,j'\}}$ it has to turn, diverting from the lowest shortest path from $i+1$ to $(j+1)'$. The first time it thus diverts gives a desired high edge. 

Now, the index of the first bad high edge cannot decrease indefinitely, and in fact one sees that the statement of the lemma holds for $M > m/2$. 
The case of edges above the highest shortest path is similar. 
\end{proof}

Define the radii $R_k$ for $k = 1,2, \ldots, m$ by
$$
R_k = \frac{\prod_{i} w_{M_i^{(k-1)} M_{i-1}^{(k)}}}{\prod_{i} w_{M_i^{(k-1)} M_{i}^{(k)}}}
$$
where we have denoted the conductance of the edge joining two vertices $v, v'$ by $w_{vv'}$.

\begin{lemma}\label{L:optimalpath}
Let $(P_{1},P_2,\ldots,P_n)$ be an $n$-tuple of consecutive shortest paths, where $P_i$ connects $i+k$ to $(i+k-1)'$ for each $i$, and some fixed $k$.  Suppose that 
\begin{enumerate}
\item
all the $P_i$ have the same shape; thus they high or low edges respectively at the same points along the path.  
\item
$R_1 \geq R_2 \geq \cdots \geq R_m$.
\end{enumerate}
Then the total weight $\wt(P_1 \cup P_2 \cup \cdots \cup P_n)$ is maximized exactly when the highest path is taken.
\end{lemma}

The weight of a subgraph is simply the product of its edge weights.

\begin{proof}
All the shortest paths are connected by switching from one side of a rhombus to the other, see Figure \ref{fig:elec3}.  
\begin{figure}[h!]
    \begin{center}
    \input{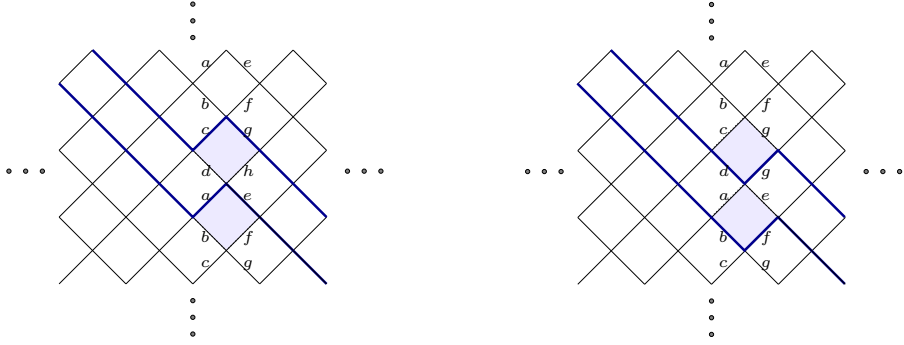}
    \end{center}
    \caption{Here $n=2$; which of the two groves has larger weight depends on which of the two radii $ac/bd$ and $fh/eg$ is bigger.}
    \label{fig:elec3}
\end{figure}
If this parallelogram involves vertices with upper index $k-1$, $k$ and $k+1$, then the higher path has greater weight exactly when $R_k > R_{k+1}$.
\end{proof}

Let $\tau_K$ be the partition of $\Z \cup \Z'$ with parts of size two $\{k,(k-1)'\}$ for all $k \in \{2,3,\ldots,K\}$, parts of size two 
$\{k,k'\}$ for $k \in \{-K,-K+1,\ldots,0\}$, and all other parts are singletons.  Thus in particular, $1$ is a singleton.  We shall denote by $\sigma_K$ the partition 
of $\Z \cup \Z'$ obtained from $\tau_K$ by placing $1$ in the same part as $\{0,0'\}$ to get a single part $\{1,0,0'\}$ of size three.

We shall suppose that the truncation $G(N)$ includes the boundary vertices $-N,-N+1,\ldots,1,2,\ldots,N$ and also the boundary vertices $-N',(-N+1)',\ldots,1,2,\ldots,N'$.  For $N \gg K$ the partitions $\tau_K$ and $\sigma_K$ naturally gives rise to partitions of the boundary vertices of $G(N)$, where boundary vertices of $G(N)$ which are not boundary vertices of $G(\infty)$ are all considered singletons.  We will still denote these partitions of boundary vertices of $G(N)$ by $\tau_K$ and $\sigma_K$.

\begin{theorem} \label{thm:asw}
Suppose that $R_1 \geq R_2 \geq \cdots \geq R_m$.  Then
$$
\lim_{K \to \infty} \lim_{N \to \infty} \frac{\uPr(\sigma_K)_{G(N)}}{\uPr(\tau_K)_{G(N)}} = a
$$
where $a$ is the weight of the high edge connected to the vertex $1$ of $G$.
\end{theorem}

\begin{proof}
Let $v$ denote the vertex connected to both $0$ and $1$.  Let $e$ be the edge joining $1$ to $v$, so that $e$ has weight $a$.  Let $e'$ denote the other edge incident to $1$.

To approximate the LHS, we shall assume we have chosen $N \gg K \gg M$.  Let $\Gamma$ be a grove with boundary partition either $\sigma_K$ or $\tau_K$.  By Lemma \ref{L:short}, the grove component $\Gamma_{K-M,(K-M-1)'}$ cannot extend either above or below the set of edges contained in shortest paths from $K-M$ to $(K-M-1)'$.  In particular, there is a bound on the number of choices of $\Gamma_{K-M,(K-M-1)'}$, not depending on $K$ or $N$.  In the following, we shall assume that we have fixed such a choice for $\Gamma_{K-M,(K-M-1)'}$.  

Let $\Gamma$ be a grove with boundary partition $\tau_K$.  By Lemma \ref{L:short}, the grove component $\Gamma_{0,0'}$ cannot extend above the (unique) shortest path from $0$ to $0'$.  In particular, the vertex $v$ must lie in $\Gamma_{0,0'}$.  Thus the edge $e$ is never used in $\Gamma$, and $\Gamma \cup \{e\}$ is a grove with boundary partition $\sigma_K$.  Also by Lemma \ref{L:short}, the grove component $\Gamma_{-M,(-M)'}$ cannot extend either above or below the (unique) shortest path from $-M$ to $(-M)'$, and therefore must be exactly this shortest path.  Similar observations hold for a grove with boundary partition $\sigma_K$.  In particular, the part of the grove above $\Gamma_{-M,(-M)'}$ and the part below are essentially independent.

Let us denote by $A_\tau$ (resp. $A_\sigma$) the set of groves with boundary partition $\tau_K$ (resp. $\sigma_K$) and by $\wt(A_\tau)$ (resp. $\wt(A_\sigma)$) the total weight of that set of groves.  

Let $\Gamma$ be a grove with boundary partition $\sigma_K$.  First we observe that all but a constant number of grove components $\Gamma_{i,(i-1)'}$ are shortest paths.  
Furthermore, as $i$ goes from $2$ to $K$, the shapes of the shortest paths are locally constant, and can only change when we encounter a grove component which is not a shortest path. 
 Furthermore, as we go down, the shape of the shortest path can only become lower.  We shall call this part of $\Gamma$ the lower half of the grove.  Note that there are $m$ 
different shapes $S_1,S_2,\ldots,S_m$ of shortest paths, listed from highest to lowest, see Figure \ref{fig:elec4}.
\begin{figure}[h!]
    \begin{center}
    \input{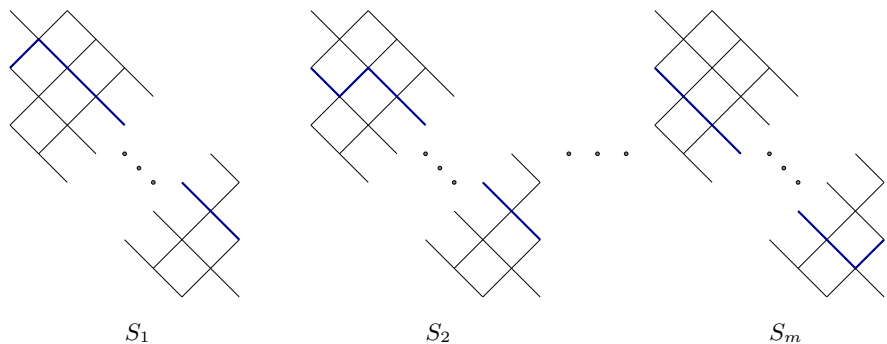}
    \end{center}
    \caption{The shapes of shortest paths connecting $i+1$ to $i'$.}
    \label{fig:elec4}
\end{figure}
Let $A'_\sigma \subset A_\sigma$ denote the subset of groves where $\Gamma_{1,0,0'}$ uses the edge $e$, and let $A''_\sigma$ denote the subset of groves where $\Gamma_{1,0,0'}$ does 
not use the edge $e'$.  Given $\epsilon > 0$, we shall now show that for sufficiently large $K$, one has $\wt(A''_\sigma)/\wt(A'_\sigma) < \epsilon$.  

Let $\Gamma \in A''_\sigma$.  Then the shortest paths which occur in the lower half of $\Gamma$ can only use the shapes $S_2,S_3,\ldots,S_m$.  Pick $K$ sufficiently large that we are 
guaranteed to have $Cnm^2/\epsilon$ grove components in the lower half which are shortest paths, where $C$ is a constant (not depending on $K$ or $N$) we shall describe below.  
Then there is some $S_i$ (say pick the least such $i$) which occurs at least $Cnm/\epsilon$ times.  Define a set of new groves $\{\Gamma'_j \mid j = 1,2,\ldots,m/\epsilon\}$ by: 
\begin{enumerate}
\item 
removing $jn$ of the $S_i$ shaped paths, and shifting the lower half of $\Gamma$ downwards to fill in the removed area;
\item
replacing the grove components $\Gamma_{-M,(-M)'},\Gamma_{-M+1,(-M+1)'},\ldots,\Gamma_{-1,(-1)'}$ with shortest paths;
\item
replacing $\Gamma_{0,0',1}$ with the shortest path from $0$ to $0'$ union the edge $e$;
\item
adding $jn-1$ new shortest paths $\Gamma'_{2,1'},\ldots,\Gamma'_{jn,(jn-1)'}$ which are $S_1$-shaped;
\item
adding one extra ``transition'' grove component $\Gamma'_{jn+1,(jn)'}$ which is $S_1$-shaped above, but which correctly fits with $\Gamma_{2,1'}$ below.
\end{enumerate}

\begin{figure}[h!]
    \begin{center}
    \input{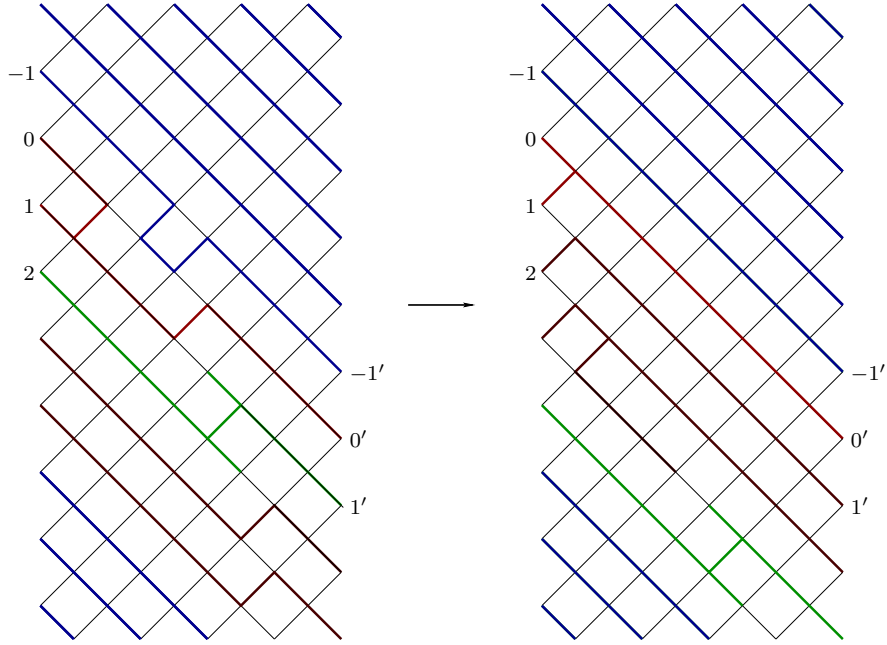}
    \end{center}
    \caption{Creation of a new grove $\Gamma'_j$ from $\Gamma$.}
    \label{fig:elec1}
\end{figure}
Figure \ref{fig:elec1} illustrates the procedure. In this case $m=9$, $n=2$, $j=1$, $i=7$. The $\Gamma_{1,0,0'}$ component of the groves is shown in red, the part shifted down is
shown in green. The two removed groves of shape $S_i$ and the two new groves of shape $S_1$ that replaced them are shown in brown.

Note that $\Gamma'_j$ belongs to $A'_\sigma$.
Since apart from the shortest paths the rest of the grove has only been modified at a bounded number of edges, by Lemma \ref{L:optimalpath} there is a constant $C_1$ not depending
 on $K, N$, or $\epsilon$ such that $\wt(\Gamma'_j) \geq C_1 \wt(\Gamma)$.  Furthermore, each grove in $A'_\sigma$ can occur in this way in at most $C_2m$ different ways.  
The $C_2$ counts the possible grove components $\Gamma_{-M,(-M)'},\Gamma_{-M+1,(-M+1)'},\ldots,\Gamma_{-1,(-1)'}$ and $\Gamma_{0,0',1'}$ of which there is a universal bound for.  
The $m$ ways count the possible $i$ such that the shortest paths $S_i$ are being replaced by $S_1$.  Setting $C = C_1 C_2$, we conclude that 
$\wt(A''_\sigma)\leq \epsilon \wt(A'_\sigma)$.  For every $\Theta \in A_\tau$, the grove $\Theta \cup \{e\}$ lies in $A_\sigma$, and this map is an injection from 
$A_\tau$ to $A_\sigma$ which changes the weight of each grove by exactly $a$.  On the other hand every $\Gamma \in A'_\sigma$ is in the image of this map.  It follows that 
$a \leq \wt(A_\sigma)/\wt(A_\tau) \leq a(1 + \epsilon)$.  Letting $K, N \to \infty$, we obtain the statement of the theorem.
\end{proof}

The partition $\tau_K$ is a planar partition of $\Z \cup \Z'$ of the type described in Theorem \ref{T:upr}(2), where all but finitely many vertices are isolated.  Define $\UPr(\tau)$ 
by taking the polynomial of Theorem \ref{T:upr}(2) and replacing $L_{ij}(G)$ by $\L_{ij}$. Then we have
$$
\UPr(\tau) = \lim_{N \to \infty}\uPr(\tau)_{G(N)}
$$ and in particular, $\UPr(\tau)$ can be approximated arbitrarily well on some $G(N)$ for very large $N$.

Unfortunately, this is not the case for the partition $\sigma_K$, which contains a part of size three, for which Theorem \ref{T:upr}(2) cannot be applied.  Instead one has
$$
\uPr(\sigma_K)_{G(N)} = p_N(L_{ij}(G(N)))
$$
for the sequence of polynomials $p_N$ of Theorem \ref{T:upr}(1).  These polynomials depend on $\sigma_K$, which is suppressed from the notation.  

\begin{prop} \label{P:polyest}
Suppose $N > N'$.  Then $p_N - p_{N'}$ is a polynomial in the $L_{ij}$'s  such that every monomial involves some $L_{ij}$ where $N \geq |i| > N'$ and $j \in [-K,K]$.  Furthermore, there is some constant $c_K$ such that the coefficient of each monomial in $p_N$ is less than $c_K$.
\end{prop}
\begin{proof}
The first statement follows from Lemma \ref{L:singles} and our choice of $\sigma_K$.  The second statement is Lemma \ref{L:combtype}. 
\end{proof}

The assymmetry of the roles of $i$ and $j$ in Proposition \ref{P:polyest} is accounted for by the fact that $L_{ij}$ is symmetric.  That is, we treat $L_{ij} = L_{ji}$ as the same variable.  Note that we already know $\lim_{N \to \infty}\uPr(\tau_K)_{G(N)}$ approaches a limit, and it follows from Theorem \ref{thm:asw} that $\lim_{N \to \infty} \uPr(\sigma_K)_{G(N)}$ approaches a limit as well.
 
\begin{lemma}\label{L:AK}
For each $K$ and $\epsilon > 0$, one can find some $A$ such that
$$
|\lim_{N \to \infty} \uPr(\sigma_K)_{G(N)} - p_{A}(\L_{ij})| < \epsilon.
$$
\end{lemma}
\begin{proof}
Fix $K$.  We first show that there is $A$ such that for all $N>A$ we have $$|p_{N}(L_{ij}(G(N))) - p_A(L_{ij}(G(N)))| < \epsilon/3.$$

It is known that the polynomials $p_N$ have degree $2K+1$, see Section \ref{sec:KWpoly}.  By Proposition \ref{P:polyest}, every monomial in $p_N - p_A$ has a factor $L_{ij}$ where $|i| > A$ and $|j| \leq K$, and has coefficient $\leq c_K$.  Thus 
$$
|p_{N}(L_{ij}(G(N))) - p_A(L_{ij}(G(N)))|  \leq \left(\sum_{|i|>A,\;|j| \leq K} L_{ij}\right) \;(\max_{a,b}(L_{ab}))^{2K}\; c_K.
$$
By the proof of Theorem \ref{T:universalresponse}, we know that for each $i$ we have $\sum_j \L_{ij} < \infty$.  Thus it is possible to find $A$ large enough that $$\left(\sum_{|i|>A,\;|j| \leq K} L_{ij}\right) < \frac{\epsilon}{3(\max_{a,b}(L_{ab}))^{2K}\; c_K},$$ giving us \begin{equation}\label{E:ineq1}|p_{N}(L_{ij}(G(N))) - p_A(L_{ij}(G(N)))| < \epsilon/3.\end{equation}

Since only finitely many $L_{ij}$'s appear in $p_A$, for sufficiently large $N$ we have \begin{equation}\label{E:ineq2}p_{A}(\L_{ij}) - p_A(L_{ij}(G(N))) < \epsilon/3.\end{equation} Furthermore, for sufficiently large $N$ we have \begin{equation}\label{E:ineq3}|\uPr(\sigma_K)_{G(N)} -\lim_{N' \to \infty} \uPr(\sigma_K)_{G(N')}| < \epsilon/3.\end{equation} Finally, we combine the three estimates \eqref{E:ineq1},\eqref{E:ineq2}, \eqref{E:ineq3} and use $\uPr(\sigma_K)_{G(N)} = p_N(L_{ij}(G(N)))$.
\end{proof}

\begin{theorem} \label{thm:main}
Fix $G$, satisfying $R_1 \geq R_2 \geq \cdots \geq R_m$ and fix one of the edges connected to one of the vertices $\ldots,-2,-1,0,1,2,\ldots$. 
There is a sequence of polynomials (depending on the universal response matrix of $G$), $p_{A_1},p_{A_2},\ldots$, and $q_1,q_2,\ldots$ such that
$$
\lim_{K \to \infty} \frac{p_{A_K}(\L_{ij})}{q_K(\L_{ij})} = a
$$
where $a$ is the weight of the chosen edge.  
\end{theorem}
\begin{proof}
By symmetry, it is enough to establish the formula for the high edge connected to $1$.  The polynomial $q_K$ is the one associated to $\tau_K$ from Theorem \ref{T:upr}(2).  The polynomial $p_{A_K}$ is chosen via Lemma \ref{L:AK} so that $|p_{A_K}(\L_{ij}) - \lim_{N \to \infty} \uPr(\sigma_K)_{G(N)}|<\epsilon_K$, where $\epsilon_K$ is chosen so that $\frac{\epsilon_K}{q_K(\L_{ij})} \to 0$ as $K \to \infty$.  Finally, we apply Theorem \ref{thm:asw}.
\end{proof}

\section{The Inverse Dirichlet-to-Neumann problem on a cylinder}
\subsection{Solution to inverse problem for the networks $N(m)$}

\begin{lemma} \label{lem:wh}
 Assume a cylindrical network $Y=N(m)$ is obtained by concatenating two networks $X=N(1)$ and $X'=N(m-1)$. Then knowing $X$ and the universal response matrix of $Y$, one can recover the universal response matrix of $X'$.
\end{lemma}

\begin{proof}
 Assume the conductances in $X$ are $a_i$ for the high edges and $b_i$ for the low edges. Concatenate $Y$ with a network $Z = N(1)$ with (virtual) conductances $-a_i$ for low edges and $-b_i$ for high edges. 
We claim that the response of the resulting network is equal to that of $X'$. Indeed, connect the opposite vertices of $Z$ and $X$ by edges with conductances $L$ and $-L$, 
without changing the response.  Changing resulting triangles into stars using the $Y-\Delta$ transformation, and letting $L \to \infty$, we see that there is an infinite conductance between opposite vertices and zero conductance between other pairs.
Thus the two $N(1)$ networks effectively cancel each other out, and we are left with a network with the same response as $X'$.  See Figure \ref{fig:elec9}.
\begin{figure}[h!]
    \begin{center}
    \input{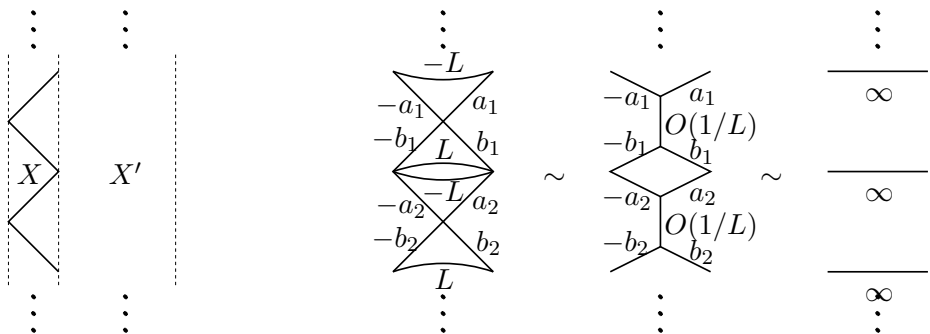}
    \end{center}
    \caption{Attaching $Z$ to $Y$ effectively removes the leftmost layer $X$.}
    \label{fig:elec9}
\end{figure}
\end{proof}

\begin{theorem} \label{thm:sol}
  There are generically $m!$ sets of edge conductances which produce a given universal response matrix for the network $N(m)$.  
All the solutions are connected by the $S_m$-action via the electrical $R$-matrix.
\end{theorem}

\begin{proof}
 It is easy to see from Theorem \ref{T:elecwhirl} that the electrical $R$-matrix swaps the radii $R_k$. Assume we have a solution for conductances in $N(m)$
 with given universal response matrix.
Apply the electrical $R$-matrix to reorder the radii $R_k$ in non-increasing order.  Then Theorem \ref{thm:main} allows us to recover the conductances in 
the leftmost $N(1)$ part of the network.  In particular, these conductances are the same for any solution.  Once we know that, we can use Lemma \ref{lem:wh} to 
recover the universal response matrix of the remaining $N(m-1)$ part of the network. Then we repeat the procedure. We see that once we require the radii to form a non-increasing sequence, 
the conductances are recovered uniquely by Theorem \ref{thm:main}. Therefore
all other solutions can be obtained from that one by action of $S_m$, which is what we want.  In the generic case when all radii are distinct, the orbit has size $m!$.
\end{proof}

\begin{remark}\label{rem:monodromy}
In the language of \cite{LP3}, Theorem \ref{thm:sol} says that the group $S_m$ generated by electrical $R$-matrices is exactly the {\it monodromy group} of the network $N(m)$.  It is the group acting on the edge weights of $N(m)$ obtained by transforming the network via local electrical equivalences back to itself.
\end{remark}

\subsection{Conjectural solution to general case}

It is convenient to describe the general answer we expect using the language of {\it {medial graphs}}, see for example \cite{CIM, dVGV}.  Draw {\it wires} through the electrical network so that 
they pass through each edge and connect inside each face as shown in the first two pictures in Figure \ref{fig:elec17}. The third picture shows an example of an electrical network and its 
medial graph.  Note that the medial graph always has four-valent vertices, and that the {\it wires} of the medial graph ``go straight through'' each vertex.

\begin{figure}[h!]
    \begin{center}
    \input{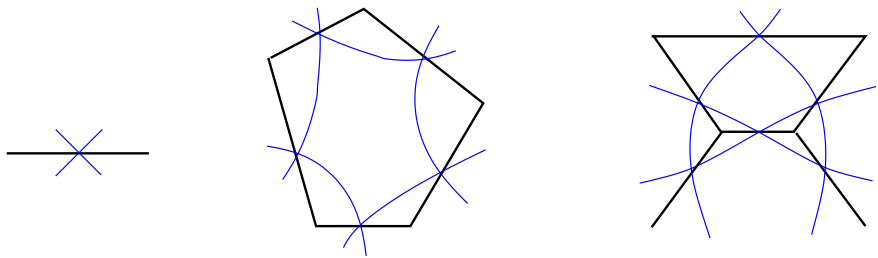}
    \end{center}
    \caption{Constructing the medial graph: the black edges belong to the electrical network, and the blue lines constitute the medial graph.}
    \label{fig:elec17}
\end{figure}

The medial graph of the networks $N(m)$ looks like $2n$ horizontal wires crossed by $m$ cycles, as shown in Figure \ref{fig:elec18}.
\begin{figure}[h!]
    \begin{center}
    \input{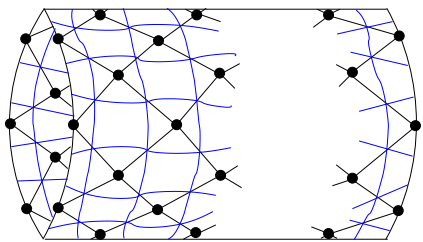}
    \end{center}
    \caption{$(N(m)$ and its medial graph.}
    \label{fig:elec18}
\end{figure}

In the terminology of \cite{CIM} a circular planar electrical network is {\it {critical}} if its medial graph avoids {\it {lenses}}, which is equivalent to saying that every pair of wires crosses as few times as possible,
given their respective homotopy types.  If $G$ is a cylindrical electrical network, we say that $G$ is {\it critical} if the universal cover of $G$ satisfies this condition; namely, the medial graph has wires which cross as few times as possible.

Let us call a cylindrical network {\it {canonical}} if the medial graph of its universal cover has the following form.  First, there are three kinds of wires: (I) some wires connect points on opposite boundaries, (II) some wires 
connect points on the same boundary, and (III) some wires do not intersect the boundary at all and correspond to (simple) cycles around the cylinder. Secondly, we require that the third kind of wires do not intersect the second kind, and furthermore, all points of intersection of wires of the first kind with themselves happen strictly before they intersect wires of the third kind. An illustration is given in Figure 
\ref{fig:elec19}.

\begin{figure}[h!]
    \begin{center}
    \input{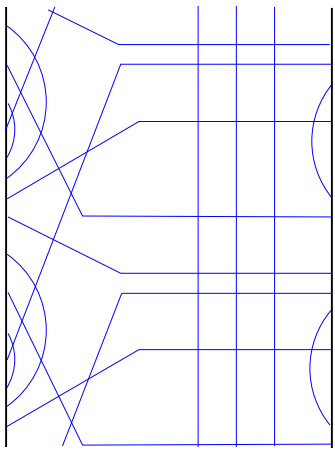}
    \end{center}
    \caption{The (universal cover of the) medial graph of a canonical network.}
    \label{fig:elec19}
\end{figure}

\begin{conjecture} \label{conj:gen}
\
\begin{itemize}
\item[(1')]
Any cylindrical electrical network can be transformed using local electrical transformations (those in Section \ref{ss:trans} and the electrical $R$-matrix) into a critical cylindrical electrical network.
\item[(2')]
Any two cylindrical electrical networks $G$ and $G'$ with the same universal response matrices are connected by local electrical equivalences.  Furthermore, if $G$ and $G'$ are both critical, then only star-triangle transformations, and electrical $R$-matrix transformations are needed.
\item[(3')]
If a cylindrical electrical network is critical canonical, then the conductances corresponding to all crossings involving wires of types (I) and (II) can be recovered uniquely. The conductances corresponding to crossings of wires of type (I) and type (III) can be recovered up to the electrical $R$-matrix action.
\item[(5')]
The space $X$ of universal response matrices of cylindrical electrical networks has an infinite stratification by $X = \sqcup C_i$ where each $C_i \simeq \R_{>0}^{d_i} \times \R_{\geq 0}^{e_i}$ is a semi-closed cell that can be obtained as the set of universal response matrices for a fixed critical network with varying edge weights. 
\end{itemize}
\end{conjecture}

For the (missing) cylindrical analogue of (4) of the Theorem in the introduction see Section \ref{ssec:TNN}.

Another way to phrase Conjecture \ref{conj:gen}(3') is that the {\it monodromy group} of a critical canonical cylindrical network is a symmetric group, generated by electrical $R$-matrices.  See Remark \ref{rem:monodromy}.

Let us explain the semi-closed cells in Conjecture \ref{conj:gen}(5').  Let $G$ be a critical canonical cylindrical electrical network.  Some edge weights can be recovered uniquely and these each give a $\R_{>0}$ in the parametrization.  The remaining part of the network is essentially one of the networks $N(m)$, whose edge weights can be recovered uniquely up to the electrical $R$-matrix action (Theorem \ref{thm:sol}).  So the response matrices would be parametrized by the orbit space $(\R_{>0}^{mn})/S_m$.  However, we can pick a distinguished element in each orbit: namely the one where the radii $R_k$ are non-increasing.  The corresponding response matrices would then be parametrized by $R_m \in \R_{>0}$, $R_1-R_2, R_2-R_3,\ldots, R_{m-1}-R_{m} \in \R_{\geq 0}$ together with some collection of edge weights which can be freely chosen in $\R_{>0}$.  Thus the universal response matrices of critical canonical cylindrical electrical network is parametrized by a semi-closed cell $C \simeq \R_{>0}^{d} \times \R_{\geq 0}^{e}$.

\end{document}